\documentclass[reqno]{amsart}
\usepackage{amsmath}
\usepackage{amssymb}
  \usepackage{paralist}
  \usepackage{graphics} 
 \usepackage[colorlinks=true]{hyperref}

\usepackage{mathrsfs}
\usepackage{extarrows}
\usepackage{enumerate}
\usepackage{float}

\newtheorem{theorem}{Theorem}[section]
\newtheorem{corollary}[theorem]{Corollary}

\newtheorem{lemma}[theorem]{Lemma}
\newtheorem{proposition}[theorem]{Proposition}

\theoremstyle{definition}
\newtheorem{definition}[theorem]{Definition}
\newtheorem{remark}[theorem]{Remark}

\newtheorem{example}[theorem]{Example}
\newcommand{\ep}{\varepsilon}

\newcommand{\RR}{\mathbb{R}}

\newcommand{\NN}{\mathbb{N}}

\newcommand{\ZZ}{\mathbb{Z}}
\newcommand{\TT}{\mathbb{T}}

\newcommand{\cC}{\mathcal{C}}

\newcommand{\cE}{\mathcal{E}}

\newcommand{\cH}{\mathcal{H}}
\newcommand{\cK}{\mathcal{K}}

\newcommand{\cN}{\mathcal{N}}
\newcommand{\cm}{\mathcal{M}}

\newcommand{\cM}{\mathcal{M}}

\newcommand{\dhh}{D_H}

\newcommand{\sL}{\mathscr{L}}
\newcommand{\pww}{P^\bot_{f,\phi}(Z)}

\newcommand{\al}{\alpha}

\newcommand{\nfw}{\cN(f,Z)}

\title[Non-dense Orbits and Approximate Product Property]
      {Non-dense orbits of systems with approximate product property}

\author[Peng Sun]{}


 \email{sunpeng@cufe.edu.cn}

\begin{document}

 \maketitle\ 


\centerline{\scshape Peng Sun}
\medskip
{\footnotesize
 \centerline{China Economics and Management Academy}
   \centerline{Central University of Finance and Economics}
   \centerline{Beijing 100081, China}
} 

\bigskip

\begin{abstract}
We show that for any topological dynamical system with approximate product property,
the  set of points whose forward orbits 
do not accumulate to any point in
a large
set carries full topological pressure.
\end{abstract}


\section{Introduction}

Let $(X,d)$ be a compact metric space and $f:X\to X$ be a continuous map.
For $x\in X$, denote the forward orbit of $x$ by
$$O_f(x):=\{f^n(x):n\in\NN\}.$$
For a subset $Z\subset X$, denote the set of points whose forward orbits 
do not accumulate to any point in
$Z$ by
$$\cN(f,Z):=\{x\in X: \overline{O_f(x)}\cap Z=\emptyset\}.$$
The points in $\cN(f,Z)$ have \emph{non-dense (forward) orbits}.
Study of such sets of non-dense orbits has motivation in homogeneous dynamics,
where it is connected to Diophantine approximation.
The Hausdorff dimensions of such sets are intensively investigated,
which sometimes led to interesting results in number theory and 
other fields.
For example, see \cite{Dani85, Dani86, Dani88,
KM96, K98, BFK, KW13, AGK15, GW, AGG, AGK20}.
Similar results are also established for more general hyperbolic or partially
hyperbolic systems \cite{Urban, Chung, Dol, HY, 
Tseng, Wu1, Wu2, Wu3}.
Non-dense orbits are also closely related to  irregular
behaviors. \cite{DT1, DT2} contain an elaborated classification of
the sets exhibiting various statistical
behaviors as well as
a multifractal analysis on them for hyperbolic systems.

In this article we illustrate a new approach, 
which studies the topological
entropy and topological
pressure carried by $\nfw$ from approximate product property, a very weak
variation of Bowen's specification property \cite{Bowen}.
We show that there is a mechanism that produces plenty of disjoint compact
$f$-invariant sets which consist of various non-dense orbits.
Approximate product property was introduced
by Pfister and Sullivan \cite{PfSu}, which is almost the weakest specification-like
property \cite{KLO, Sunintent, Sununierg}.
While Bowen's original specification property requires strong hyperbolicity,
approximate product property 
is compatible with certain non-hyperbolic behaviors.
We perceive that systems
with approximate product properties (APP systems for short) have delicate structures
in many senses
and the author has obtained some interesting results
\cite{Sunintent, Sununierg}.



Let $\phi:X\to\RR$ be a continuous potential function.
For any subset $Y\subset X$, denote 
by $P(Y,f,\phi)$ and $h(Y,f)=P(Y,f,0)$
the topological pressure and the topological entropy on $Y$.
Denote by $P(f,\phi):=P(X,f,\phi)$ and $h(f):=h(X,f)$ the topological pressure
and topological entropy of the system. We state the main result of the article as follows.

\begin{theorem}\label{thmain}
Let $(X,f)$ be 
an APP system with
positive topological entropy 
and $\phi:X\to\RR$ be a continuous potential function.
Suppose that $Z=\bigcup_{i=1}^n Z_i$ is a finite union of subsets of $X$
such that for each $i=1,\cdots, n$, 
one of the following holds:
\begin{enumerate}
\item $Z_i$ is any single forward orbit.
\item For any given $\mu_i\in\cm(X,f)$,
$Z_i$ consists of all points whose empirical measures accumulate to $\mu_i$.
In particular, $Z_i$ may contain all generic points for $\mu_i$.
\item $Z_i$ is any compact $f$-invariant subset of $X$
such that $\cm(Z_i,f)\ne\cm(X,f)$.
\item $Z_i$ consists of the points with weak $K_i$-behavior, where
$K_i$ is  
any compact subset 
of a proper weak face (see Definition \ref{weakface} and \ref{kbeh}). 
\end{enumerate}
Then
$P(\nfw,f,\phi)=P(f,\phi)$. In particular, $h(\nfw,f)=h(f)$.
\end{theorem}

The case of $Z$ in Theorem \ref{thmain} is just a noteworthy instance
but not all of them.
Our key result is Theorem \ref{propmain}. 
In this article we adopt Pesin-Pitskel's definition of topological pressures
on non-compact
sets. See \cite{PP} or \cite{Pes} for details. 
Results in this article 
remain valid if another definition 
(e.g. by $(n,\ep)$-separated sets \cite{Walters})
of $P(Y,f,\phi)$
 is adopted
as long as it coincides with Pesin-Pitskel's definition whenever 
$Y$ is compact
and $f$-invariant, i.e. $f(Y)\subset Y$.
They also remain valid if $\phi$ is replaced by
an asymptotically additive potential $\Phi$ 
introduced in \cite{FH}, as long as
the integral of 
the potential is continuous. Moreover, by
\cite{Sununierg}, an APP system with 
zero topological
entropy must be uniquely ergodic. In this case every point is generic for
the unique ergodic measure hence $\nfw$ may be empty.
Finally, we notice 
 that \cite{Zhao} contains a similar result for the case
that the system has specification property and $Z$ consists of just a single
non-transitive point.

APP systems form a broad class that includes most familiar systems.
The following provides an incomplete list of them, to which our results
apply:
\begin{enumerate}
\item
Transitive sofic shifts;
\item $\beta$-shifts;


\item Ergodic toral automorphisms;
\item Transitive graph maps;
\item A homogeneneous system $(G/\Gamma,g)$, 
where $G$ is connected semisimple Lie group without compact factors,
$\Gamma$ is an irreducible cocompact lattice of G and $g\in G$ is non-quasiunipotent
\cite{GSW};

\item 
Every $C^0$-generic map $f$ on a compact
Riemannian manifold restricted to every 
chain-recurrent class $\cC$ for $f$ \cite{BeTV};

\item Certain partially hyperbolic diffeomorphisms, 
e.g. transitive time-1 maps of Anosov flows;

\item A product of an APP system and a system
with tempered specification property,
e.g. the product of an irrational rotation and an ergodic toral automorphism;
\item Factors and conjugates of above systems.

\end{enumerate}

Note that for  symbolic systems Theorem \ref{thmain}
and Theorem \ref{propmain} directly
yield the corresponding results on the Hausdorff dimension of $\nfw$,
which generalize \cite[Theorem 1]{Dol}.
Moreover, in the above list there are certain homogeneous systems.
Our results for these systems are in some sense related to the
conjectures of Margulis \cite{Marg}.

Theorem \ref{thmain} 
 is a direct corollary of  Theorem \ref{propmain} and Proposition
\ref{propunion}. Note that among the cases of $Z_i$ in Theorem \ref{thmain},
Case (2) covers Case (1) and
 by Lemma \ref{lemcomp}, 
Case(4) covers Case (3).
We shall introduce our notations in Section 2. Then
we focus on APP systems in Section 3. Finally we explain the other notions
involved in Theorem \ref{thmain} and complete the proof of the theorem in the
last section.

\section{Notations}

Let $(X,f)$ be a topological dynamical system.
Denote 
by $\cM(X)$ the space of probability measures on $X$, by $\cm(X,f)$ 
the subspace of all invariant probability measures for
$(X,f)$ and by $\cm_e(X,f)$
the subset consisting of the ergodic ones. 
As $X$ is compact, both $\cM(X)$ and $\cm(X,f)$ are  compact metrizable spaces 
in the weak-$*$
topology \cite[Theorem 6.5 and Theorem 6.10]{Walters}.

Denote by $D$ a metric on $\cm(X)$ that induces the weak-$*$ topology on $\cm(X)$.
Denote
$$\cK(X,f):=\{K: \text{$K$ is compact subset of $\cm(X,f)$}\}.$$
Then $\cK(X,f)$ is a compact metric space with the Hausdorff metric
$$D_H(K_1,K_2):=\max\{\max_{\mu\in K_1}\min_{\nu\in K_2}D(\mu,\nu),
\max_{\nu\in K_2}\min_{\mu\in K_1}D(\mu,\nu)\}.$$

Denote by 
$\ZZ^+$ the set of all
positive integers. 
For $x\in X$ and $n\in\ZZ^+$, we define the \emph{empirical measure} $\cE(x,n)$ such that
$$\int\phi d\cE(x,n):=\frac1n\sum_{k=0}^{n-1}\phi(f^k(x))\text{ for every
}\phi\in C(X).$$
Denote
$$\Omega(x):=\left\{\mu\in\cm(X):\mu\text{ is a weak-$*$ 
accumulation point of $\left\{\cE\left(x,n\right)\right\}_{n=1}^\infty$
}\right\}.$$
Then 
every $\mu\in\Omega(x)$ is an invariant measure and $\Omega(x)$ is closed.
Hence $\Omega(x)\in\cK(X,f)$ for every $x\in X$.


Denote by $h_\mu(f)$ the metric entropy of $(X,f)$ with respect to
$\mu\in\cm(X,f)$ and by $P_\mu(f,\phi):=h_\mu(f)+\int\phi d\mu$ the
pressure of $\mu$. When $Y$ is a compact $f$-invariant set,
the topological entropy and topological pressure can be calculated with
$(n,\ep)$-separated subsets of $Y$ and we denote by
$h(Y,f,\ep)$ and $P(Y,f,\phi,\ep)$ their values at the scale $\ep$.
It holds that
\begin{equation*}
P(Y, f,\phi)=\sup\{P(Y,f,\phi,\ep):\ep>0\}
=\sup\left\{P_\mu(f,\phi): \mu\in\cm(Y,f)\right\}.
\end{equation*}

Readers are referred to the books \cite{Pes} and \cite{Walters} for more
details on measures, entropy and pressure.

\section{Approximate Product Property}

\begin{definition}\label{defapp}
The system $(X,f)$ is said to have \emph{approximate product property},
or called an \emph{APP system},
if for every $\ep, \delta_1, \delta_2>0$, 
        there is $N
        >0$
        such that for every $n\ge N$ and every sequence $\{x_k\}_{k=1}^\infty$ in $X$, there exist an 
        increasing sequence $\{s_k\}_{k=1}^\infty$ of integers and $z\in X$ such that 
        $$s_1=0\text{ and }n\le s_{k+1}-s_k<n(1+\delta_1)\text{ for each }k\in\ZZ^+,\text{ and}$$
        \begin{equation*}
        \left|\left\{0\le j\le n-1: d(f^{s_{k}+j}(z), f^j(x_k))>\ep\right\}\right|<\delta_2
        n\text{ for each $k\in\ZZ^+$}.
        \end{equation*}
\end{definition}

Approximate product property is almost the weakest specification-like property.
It is weaker than almost specification property (also called $g$-almost product property),
tempered specification property (also called almost weak specification property
or weak specification property), gluing orbit property, etc.
More detailed discussions on specification-like properties can be found 
in \cite{DGS}, \cite{KLO},
\cite{Sunintent} and \cite{Sununierg}.


The following is an essential fact 
for APP systems, which is an improved version
of {\cite[Proposition 2.3 and Theorem 2.1]{PfSu}}.

\begin{proposition}[{\cite[Proposition 5.1]{Sunintent}}]
\label{entropydense}
Let $(X,f)$ be an APP system.
Then for any $\mu\in\cm(X,f)$, any $h\in\left(0,h_\mu(f)\right)$ and
any $\eta, \ep, \delta_0>0$, there are 
$\delta\in(0,\delta_0)$ and a compact $f$-invariant subset 
$\Lambda=\Lambda(\mu, h, \eta, \ep, \delta)$ such that
\begin{enumerate}
\item
There is $N\in\ZZ^+$ such that $D(\cE(x,n),\mu)<\eta$
for every $x\in\Lambda$ and every $n>N$.

\item $h<h(\Lambda,f,\delta)<h+\ep$. In particular, 
$h(\Lambda,f)>h.$
\end{enumerate}
\end{proposition}

\begin{corollary}\label{corpres}
Let $(X,f)$ be an APP system and $\phi$ be a continuous
potential. 
Then for any $\mu\in\cm(X,f)$, any $\al\in\left(\int\phi d\mu,P_\mu(f,\phi)\right)$ and
any $\eta, \ep, \delta_0>0$, there are 
$\delta\in(0,\delta_0)$ and a compact $f$-invariant subset 
$\Lambda=\Lambda_\phi(\mu,\al, \eta, \ep, \delta)$ such that
\begin{enumerate}
\item
There is $N\in\ZZ^+$ such that $D(\cE(x,n),\mu)<\eta$
for every $x\in\Lambda$ and every $n>N$.
\item $\al<P(\Lambda,f,\phi,\delta)<\al+\ep$. In particular, 
$P(\Lambda,f,\phi)>\al.$
\end{enumerate}
\end{corollary}

\begin{proof}
By continuity of $\phi$, there is $\eta'\in(0,\eta)$ such that
\begin{equation}\label{eqintest}
|\int\phi d\nu-\int\phi d\mu|<\frac\ep3\text{ whenever }D(\nu,\mu)<\eta'.
\end{equation}
We may assume that $\al+\ep\le P_\mu(f,\phi)$. Then
$$\al-\int\phi d\mu+\frac\ep3<h_\mu(f).$$
Let $\delta\in(0,\delta_0)$ and 
$$\Lambda:=\Lambda\left(\mu, \al-\int\phi d\mu+\frac\ep3,
\eta', \frac\ep3, \delta\right)$$ 
be as obtained 
from Proposition \ref{entropydense}.
Then Condition (1) in Corollary \ref{corpres} is satisfied
as $\eta'<\eta$. Moreover, 
by \eqref{eqintest},
we have
\begin{align*}
P(\Lambda,f,\phi,\delta)
&\ge
h(\Lambda, f,\delta)
+\inf\left\{\int\phi d\nu: \nu\in\cm(\Lambda,f)\right\}
\\&>\left(\al-\int\phi d\mu+\frac\ep3\right)
+\left(\int\phi d\mu-\frac\ep3\right)
\\&
=\al
\end{align*}
and
\begin{align*}
P(\Lambda,f,\phi,\delta)
&\le
h(\Lambda, f,\delta)
+\sup\left\{\int\phi d\nu: \nu\in\cm(\Lambda,f)\right\}
\\&<\left(\al-\int\phi d\mu+\frac\ep3\right)+\frac\ep3
+\left(\int\phi d\mu+\frac\ep3\right)
\\&
=\al+\ep.
\end{align*}

\end{proof}

\begin{remark}\label{muclose}
Note that the first conditions in Proposition \ref{entropydense}
and Corollary \ref{corpres} 
imply that for every $x\in\Lambda$, we have
$$D(\mu,\nu)\le\eta\text{ for every $\nu\in\Omega(x)$},$$
and 
hence
$\dhh(\{\mu\},\Omega(x))\le\eta$.
\end{remark}



Let $Z$ be a subset of $X$. Denote
$$\Omega(Z):=\{\Omega(x): x\in Z\}\subset\cK(X,f).$$
and
$$\pww:=\sup\left\{P_\mu(f): \{\mu\}\notin
\overline{\Omega(Z)},\mu\in\cm(X,f)\right\},$$
where the closure of $\Omega(Z)$ is taken with respect to 
the $\dhh$ metric on $\cK(X,f)$.
In particular, we put $\pww:=0$ if $\overline{\Omega(Z)}=\cm(X,f)$.

\begin{theorem}\label{propmain}
Let $(X,f)$ be an APP system with positive topological entropy and
$\phi$ be a continuous potential. 
Then for any subset $Z$ of $X$, we have
$$P(\cN(f,Z),f,\phi)\ge\pww.$$
\end{theorem}

\begin{proof}
        If $\overline{\Omega(Z)}=\cm(X,f)$ then $\pww=0$. The result is trivial.
        
Otherwise, for every 
$\al<\pww$, there is $\mu\in\cm(X,f)$ such that
$$\{\mu\}\notin\overline{\Omega(Z)}\text{ and }P_\mu(f)>\al.$$
As $\overline{\Omega(Z)}$ is compact, there is
$\eta>0$ such that
\begin{equation}\label{eqnunotinz}\eta<
\min\left\{\dhh\left(\{\mu\},K\right):K\in\overline{\Omega(Z)}\right\}.
\end{equation}


Take any $\ep>0$.
By 
Corollary \ref{corpres} and Remark \ref{muclose}, 
there is a compact $f$-invariant set $\Lambda$
such that
\begin{equation}\label{eqclose}
\dhh(\{\mu\},\Omega(x))\le\eta
\text{ for every }x\in\Lambda
\end{equation}
and
$$P(\Lambda,f,\phi)>P_{\mu}(f,\phi)-\ep.$$





We claim that $\Lambda\cap Z=\emptyset$.
Suppose that $y\in\Lambda\cap Z$. By \eqref{eqclose},
$y\in\Lambda$ implies that
$$\dhh\left(\{\mu\},\Omega(y)\right)\le\eta.$$
But by \eqref{eqnunotinz}, $y\in Z$ implies that
$$D_H(\{\mu\},\Omega(y))\ge 
\min\left\{\dhh\left(\{\mu\},K\right):K\in\overline{\Omega(Z)}\right\}>\eta.$$
This is a contradiction.

As 
$\Lambda$ is compact and $f$-invariant, 
we have $\overline{O_f(x)}\subset\Lambda$ for every $x\in\Lambda$.
This implies that
$\Lambda\subset\cN(f,Z)$. Then
$$P(\cN(f, Z),f,\phi)\ge P(\Lambda,f,\phi)
>P_{\mu}(f,\phi)-\ep>\al-\ep.$$
As $\al<\pww$ and $\ep>0$ are arbitrarily taken, we have 
$$P(\cN(f, Z),f,\phi)\ge\pww.$$
\end{proof}

Let $\mu\in\cm(X,f)$ and $\eta>0$. Denote
$$B(\mu,\eta):=\{\nu: D(\nu,\mu)<\eta\}.$$
If $\Omega(x)\nsubseteq B(\mu,\eta)$ for every $x\in Z$, then
$\{\mu\}\notin\overline{\Omega(Z)}$ and $\pww\ge P_{\mu}(f,\phi)$.
The following is a direct corollary of Theorem \ref{propmain}.

\begin{corollary}
Let $(X,f)$ be an APP system with positive topological entropy and
$\phi$ be a continuous potential.  Let $\mu\in\cm(X,f)$. Suppose that
 there is
$\eta>0$ such that $\Omega(x)\nsubseteq B(\mu,\eta)$ for every $x\in Z$.
Then
$$P(\cN(f, Z),f,\phi)\ge P_{\mu}(f,\phi).$$
In particular, if $\mu$ is an equilibrium state of $(X,f,\phi)$, then
$$P(\cN(f, Z),f,\phi)=P(f,\phi).$$
\end{corollary}

We remark that it is possible that $\pww=0$ when $Z$ is countable.
For example, suppose that the system $(X,f)$ has periodic tempered gluing orbit
property (e.g. a quasi-hyperbolic toral automorphism). 
Then the ergodic measures supported on periodic orbits are dense in $\cm(X,f)$, 
hence $\overline{\Omega(Z)}=\cm(X,f)$ if
$Z$ is the countable set consisting of all periodic points.
This case is beyond the limitation of our approach.

\section{Weak Faces} 

\begin{definition}[{cf. \cite{CTV}}]\label{weakface}
A convex subset $L$ of $\cm(X,f)$ is called a \emph{weak face}
if for any $\mu\in L$, $\mu=\lambda\nu_1+(1-\lambda)\nu_2$ for
$\lambda\in(0,1)$ and $\nu_1,\nu_2\in\cm(X,f)$ implies that
$\nu_1,\nu_2\in L$. We say that $L$ is \emph{proper} if $L\ne\cm(X,f)$.
\end{definition}

\begin{remark}
Existence of nonempty proper weak face requires that $(X,f)$
is not uniquely ergodic.
\end{remark}

\begin{lemma}
Let $\sL=\{L_\theta\}_{\theta\in I}$ be any family of weak faces.
Then both
$\bigcup_{\theta\in I} L_\theta$ and $\bigcap_{\theta\in I} L_\theta$
are weak faces.
\end{lemma}

\begin{proof}
Let $\mu\in\bigcup_{\theta\in I} L_\theta$, $\mu=\lambda\nu_1+(1-\lambda)\nu_2$ for
$\lambda\in(0,1)$ and $\nu_1,\nu_2\in\cm(X,f)$. 
Then there is $\theta_0$ such
that $\mu\in L_{\theta_0}$. As $L_{\theta_0}$ is a weak face, we must have
$$\nu_1,\nu_2\in L_{\theta_0}\subset\bigcup_{\theta\in I} L_\theta.$$
So $\bigcup_{\theta\in I} L_\theta$ is a weak face.

Let $\mu\in\bigcap_{\theta\in I} L_\theta$, $\mu=\lambda\nu_1+(1-\lambda)\nu_2$ for
$\lambda\in(0,1)$ and $\nu_1,\nu_2\in\cm(X,f)$. 
For each $\theta\in I$, we have $\mu\in L_{\theta}$ and $L_{\theta}$ is a weak face,
hence
$\nu_1,\nu_2\in L_\theta$. 
This implies that
$\nu_1,\nu_2\in \bigcap_{\theta\in I} L_\theta$.
So $\bigcap_{\theta\in I} L_\theta$ is a weak face.
\end{proof}

\begin{lemma}\label{lemproper}
If $L=\bigcup_{i=1}^\infty L_i$ and each $L_i$ is a proper weak face, then
$L$ is proper.
\end{lemma}

\begin{proof}
Fix any $\mu_0\in\cm(X,f)$.
For each $i$, take $\mu_i\in\cm(X,f)\backslash L_i$. Let
$$\mu_n:=
(1-2^{-n})\mu_0+{\sum_{i=1}^n 2^{-i}\mu_i}\in\cm(X,f).$$
Then $\{\mu_n\}_{n=1}^\infty$ converges to some $\mu\in\cm(X,f)$. 
For each $i$ and each $n>i$,
we can write
$$\mu_n=(1-2^{-i})\nu_{i,n}+2^{-i}\mu_i,$$
where
$$\nu_{i,n}:=\frac{1}{1-2^{-i}}\left((1-2^{-n})\mu_0+
\sum_{j\in\{1,\cdots,n\}\backslash\{i\}} 2^{-j}\mu_j\right)\in\cm(X,f).$$
Then $\{\nu_{i,n}\}_{n=1}^\infty$ converges to 
some $\nu_i\in\cm(X,f)$ and 
$$\mu=(1-2^{-i})\nu_{i}+2^{-i}\mu_i.$$
This implies that $\mu\notin L_i$ as $\mu_i\notin L_i$ for each $i$.
So $\mu\notin L=\bigcup_{i=1}^\infty L_i$, hence $L\ne\cm(X,f)$.
\end{proof}

\begin{corollary}
Let $K:=\bigcup_{i=1}^n K_i$ such that each $K_i$ is a compact subset of
a proper weak face. Then $K$ is also a compact subset of a proper weak face.
\end{corollary}

\begin{definition}\label{kbeh}
Let $x\in X$ and  $K$ be a subset of $\cm(X,f)$. 
We say that
$x$ has \emph{weak $K$-behavior} if $\Omega(x)\cap K\ne\emptyset$.
We denote by by $\cH(K)$
the set consisting of all points with weak $K$-behavior.
\end{definition}


Following \cite{CTV}, 
we say that
$x$ has \emph{$K$-behavior} if $\Omega(x)\subset K$,
and $x$ is a point \emph{without $K$-behavior} if $\Omega(x)\subset 
\cm(X,f)\backslash K$
(we are aware that this notion is a bit misleading).
By definition, 
$x$ is a point without 
$K$-behavior
 if and only if $x\notin\cH(K)$. 
In particular, we have $\cN(f,\cH(K))\subset \cH(K)^c$,
i.e. every $x\in\cN(f,\cH(K))$ is a point without $K$-behavior.





\begin{definition}\label{defmc}
A subset $C$ of $X$ is called the \emph{measure center} of the system
$(X,f)$ if $C$ is the smallest closed subset such that
$\mu(C)=1$ for any $\mu\in\cm(X,f)$.
\end{definition}

\begin{lemma}
For any system $(X,f)$ and any compact $f$-invariant subset $Y$, 
the following are equivalent:
\begin{enumerate}
\item $\cm(Y,f)\ne\cm(X,f)$.
\item $Y$ does not include the measure center of $(X,f)$.
\item There is $\mu\in\cm(X,f)$ such that $\mu(Y)<1$.
\end{enumerate}
\end{lemma}

\begin{lemma}\label{lemcomp}
If $Y$ is a compact $f$-invariant subset
such that $\cm(Y,f)\ne\cm(X,f)$, then $Y\subset\cH(\cm(Y,f))$
and $\cm(Y,f)$ is a compact proper weak face.
\end{lemma}

\begin{proof}
As $Y$ is compact and $f$-invariant, $(Y,f)$ is a subsystem.
Hence $\cm(Y,f)$ is a compact subset of $\cm(X,f)$ and it is a weak face.
For every $x\in Y$ we have $\Omega(x)\subset\cm(Y,f)$. This implies that
$Y\subset\cH(\cm(Y,f))$. As $Y$ does not include the measure center of $(X,f)$,
we have $\cm(Y,f)\ne\cm(X,f)$. So $\cm(Y,f)$ is a compact proper weak face.
\end{proof}

We say that $x$ is a \emph{generic} point for $\mu\in\cm(X,f)$ if $\Omega(x)=\{\mu\}$.
Note that $\mu$ is not necessarily an ergodic measure to have generic points.
The singleton $\{\mu\}$ may not be included in a proper weak face.
In Case (2) of Theorem \ref{thmain},
$Z_i$ consists of all points whose empirical measures accumulate to $\mu_i$
if and only if $Z_i=\cH(\{\mu_i\})$. But this is not covered by Case (4)
in the theorem. In the following proposition we consider the two cases
separately.


\begin{proposition}
\label{propunion}
Suppose that $(X,f)$ is not uniquely ergodic.
Let 
$U:=\bigcup_{i=1}^m U_i$ 
such that for each $i$ we have $U_i=\cH(\{\mu_i\})$ 
for an invariant measure 
$\mu_i\in\cm(X,f)$.
Let $V:=\bigcup_{j=1}^n \cH(K_j)$
such that each
$K_j$
is a compact subset of a proper weak face $L_j$.
Let $Z:=U\cup V$. 
Then  $\pww=P(f,\phi)$.
\end{proposition}

\begin{proof}

Take any $\al<P(f,\phi)$. 
There is an invariant measure
$\mu$ such that $P_\mu(f,\phi)>\al$. 
Let $L:=\bigcup_{j=1}^nL_j$. 
As $(X,f)$ is not uniquely ergodic, by Lemma \ref{lemproper},
we can find $\mu_0\in\cm(X,f)\backslash L$ such that $\mu_0\ne\mu$. 
Then there is $\theta\in(0,1)$ such that
for $\mu':=\theta\mu+(1-\theta)\mu_0$ we have 
$$P_{\mu'}(f,\phi)>\al\text{ and }\mu'\notin\{\mu_i:i=1,\cdots,m\}.$$ 
For each $j$, as $L_j$ is a weak face and either $\mu\notin L_j$ or $\mu_0\notin L_j$ holds,  we must have 
$\mu'\notin L_j$. Hence $\mu'\notin L$.

Let
$$K:=\{\mu_i: i=1,\cdots, m\}\cup\left(\bigcup_{j=1}^n K_j\right).$$
Then $K$ is compact and $\mu'\notin K$.
There is $\eta>0$ such that
$D(\mu',\nu)>\eta$ for every $\nu\in K$.

For every $x\in U$, we have $\mu_i\in\Omega(x)$ for some $i$. Then
$$\dhh\left(\{\mu'\},\Omega(x)\right)\ge D(\mu',\mu_i)>\eta.$$
For every $x\in V$, we have $\Omega(x)\cap K_j\ne\emptyset$ for some $j$.
Then
$$\dhh(\{\mu'\},\Omega(x))\ge\min\{D(\mu',\nu):\nu\in K_j\}>\eta.$$
So 
$$\dhh(\{\mu'\},\Omega(x))>\eta\text{ for every }x\in Z.$$
This implies that $\{\mu'\}\notin\overline{\Omega(Z)}$.
Then $$\pww\ge P_{\mu'}(f,\phi)>\al.$$
As $\al$ is arbitrary,
we have $\pww=P(f,\phi)$.
\end{proof}

By \cite{Sununierg}, APP systems with positive topological entropy are not uniquely ergodic.
So Proposition \ref{propunion} holds for such systems and verifies Theorem
\ref{thmain} based on Theorem \ref{propmain}.


The following example 
provides a motivation of our consideration of compact subsets of 
weak faces.
It also indicates that our approach may have
more applications.

\begin{example}
Let $M$ be a compact Riemannian manifold, $f:M\to M$ be a
$C^1$ diffeomorphism with a dominated splitting $TM=E\oplus F$
and $(M,f)$ is an APP system.
Assume that the Lyapunov exponents are non-positive along $E$ 
and non-negative along $F$.
Then by \cite{CCE} and \cite{CTV},
the set $PL(f)$ consisting of all physical-like measures
is a compact subset of a weak face
(the set consisting of all invariant measures satisfying Pesin entropy formula).
In this setting we are able to generalize the results in \cite{CTV}.
A simple example of such a system that is not covered by \cite{CTV} is
 the product of an irrational rotation and a quasi-hyperbolic toral automorphism.
Note that when $Z=\cH\left(PL\left(f\right)\right)$ is the set of all points with weak $PL(f)$-behavior,
the points in $\nfw$ are not just without physical-like behavior but
also have the forward orbits that do not accumulate to any point in $Z$.
\end{example}

\section*{Acknowledgments}
This work is supported by National Natural Science Foundation of China (No. 11571387)
and CUFE Young Elite Teacher Project (No. QYP1902). The author would like to
thank Jinpeng An for fruitful discussions and would also like to thank Xueting Tian and Ercai Chen for helpful comments.



\end{document}